\documentclass[11pt]{amsart}

\usepackage{amsmath,amssymb,amsthm}
\usepackage{mathtools}
\usepackage[margin=1in]{geometry}
\usepackage[hidelinks]{hyperref}
\usepackage{enumitem}
\usepackage{booktabs}
\usepackage{array}

\hypersetup{
  pdftitle={Fixed points and exact periods of Chebyshev polynomials modulo odd prime powers},
  pdfauthor={Chatchawan Panraksa and Aram Tangboonduangjit},
  pdfkeywords={Chebyshev polynomial; prime-power modulus; fixed-point lifting; exact period; arithmetic dynamics; p-adic dynamics}
}

\makeatletter
\@ifundefined{subjclassname@2020}{%
  \@namedef{subjclassname@2020}{\textup{2020} Mathematics Subject Classification}%
}{}
\makeatother

\theoremstyle{plain}
\newtheorem{theorem}{Theorem}[section]
\newtheorem{lemma}[theorem]{Lemma}
\newtheorem{proposition}[theorem]{Proposition}
\newtheorem{corollary}[theorem]{Corollary}

\theoremstyle{definition}
\newtheorem{definition}[theorem]{Definition}
\newtheorem{example}[theorem]{Example}

\theoremstyle{remark}
\newtheorem{remark}[theorem]{Remark}

\newcommand{\Z}{\mathbb{Z}}
\newcommand{\Q}{\mathbb{Q}}
\newcommand{\Fp}{\mathbb{F}_p}
\newcommand{\Fpp}{\mathbb{F}_{p^2}}
\newcommand{\nup}{\nu_p}
\newcommand{\cord}{\operatorname{ord}^{\pm}}
\newcommand{\Fix}{\operatorname{Fix}}
\newcommand{\Mob}{\mu_{\mathrm{Mob}}}

\title[Fixed points and exact periods of Chebyshev polynomials]{Fixed points and exact periods of Chebyshev polynomials modulo odd prime powers}

\author{Chatchawan Panraksa}

\author{Aram Tangboonduangjit}
\address{Mahidol University International College, Mahidol University, Nakhon Pathom 73170, Thailand}
\email[Chatchawan Panraksa]{chatchawan.pan@mahidol.ac.th}
\email[Aram Tangboonduangjit]{aram.tan@mahidol.ac.th}

\thanks{Corresponding author: Aram Tangboonduangjit, \texttt{aram.tan@mahidol.ac.th}.}

\date{}



\begin{document}
\raggedbottom

\begin{abstract}
Passing from a prime modulus to a higher prime power can change both the lengths and the number of cycles of a polynomial map. We determine these changes for Chebyshev polynomials $T_n$ modulo $p^k$, for every degree $n\ge2$, every odd prime $p$, and every $k\ge1$. The results apply whether or not $T_n$ permutes the residue classes. We give explicit formulas for the number of fixed points, all possible cycle lengths, and the number of cycles of each length. The fixed-point formula involves four greatest common divisors and a correction specific to $p=3$. We explain this exception by showing how distinct rational fixed points become congruent modulo~$3$. When $p\mid n$, each cycle modulo $p$ corresponds to exactly one cycle of the same length modulo every $p^k$. When $p\nmid n$, we determine how many cycles arise from each cycle modulo $p$ and when longer cycles first appear. The proofs combine the classical relation between Chebyshev polynomials and power maps with $p$-adic arithmetic.
\end{abstract}

\maketitle

\section{Introduction}\label{sec:intro}

A polynomial map can have several cycles modulo a prime power that reduce to the same cycle modulo the underlying prime. These cycles may have different lengths, and their behavior at higher powers depends on more than the map over the prime field: derivatives and higher-order terms also enter the lifting problem. Understanding this dependence is a basic question in arithmetic dynamics; see~\cite{Narkiewicz1995,Silverman2007}. We study it for Chebyshev polynomials, determining the fixed-point counts and the lengths and numbers of cycles over every odd prime-power ring, without assuming that the polynomial permutes the residue classes.

Throughout, $p$ is an odd prime, $n\ge2$ is the degree being iterated, and $k\ge1$ is the exponent in the modulus $p^k$. We write $\Fp=\Z/p\Z$ for the field with $p$ elements and $\Fpp$ for its quadratic extension. A \emph{lift} of $a\in\Fp$ modulo $p^k$ is a residue class whose reduction modulo $p$ is~$a$. We use the Chebyshev polynomials of the first kind, $T_m\in\Z[x]$, characterized by $T_m(\cos\theta)=\cos(m\theta)$ for $m\ge0$. Two identities make their dynamics especially accessible. The composition law $T_m\circ T_n=T_{mn}$ gives $T_n^{\circ j}=T_{n^j}$ for $j\ge1$, where the superscript $\circ j$ denotes the $j$th iterate. Also, the \emph{averaging map}
\[
  F(z)=\frac{z+z^{-1}}2
\]
satisfies $T_m(F(z))=F(z^m)$. This identity holds for units in any ring in which $2$ is invertible. Thus iteration of a Chebyshev polynomial is related to a power map, and its periodic points can be counted through fixed points of other Chebyshev polynomials.

Over $\Fp$, the relevant parameters lie in two groups:
\[
  \Fp^\times\qquad\text{and}\qquad
  \mu_{p+1}=\{\zeta\in\Fpp^\times:\zeta^{p+1}=1\}.
\]
The latter is the norm-one group. Every $a\in\Fp$ has the form $F(\zeta)$ for a parameter in their union, and its parameters are determined up to inversion; Proposition~\ref{prop:two-source} recalls the proof. We call these groups the \emph{source groups}, such a $\zeta$ a \emph{source parameter} of~$a$, and its multiplicative order $e$ the \emph{source order} of~$a$. Inversion preserves order, so $e$ is well defined. The parameters $\zeta=\pm1$ and their images $a=\pm1$ are called \emph{boundary}; all other parameters and residues are \emph{nonboundary}. The boundary parameters are exactly the zeros of the derivative of $F$.

The finite-field dynamics is described in earlier work. Gassert~\cite{Gassert2014} treated prime degree, Qureshi and Panario~\cite{QP2019} described arbitrary degree, and Hutz and Patel~\cite{HutzPatel2022} obtained related periodic-point formulas. The finite-field counts used here are consequences of this theory. Some references use the Dickson polynomial $D_n(x,1)$, related to our normalization by $D_n(2x,1)=2T_n(x)$. Since $p$ is odd, multiplication by~$2$ conjugates these maps over $\Fp$ and over $\Z/p^k\Z$.

For general polynomial maps, desJardins and Zieve~\cite[Sec.~3]{desJardinsZieve2001} developed the affine description of cycle lifting, and Nara~\cite[Theorem~4.3 and Corollary~4.12]{Nara2025} describes how derivatives govern first and successive lifts. For Chebyshev maps, Yoshioka~\cite{Yoshioka2018} studied periods of individual iterated sequences, while Li, Lu, Tan, and Chen~\cite{LLTC2025} determined cycle distributions for permutation degrees when $p>3$. Tan and Li~\cite{TanLi2025} consider related lifting questions for a broader class of permutation maps. Lu, Dai, and Li~\cite{LuDaiLi2026} treat Chebyshev permutation graphs modulo powers of $2$ and $3$. Our aim is to determine the periodic part of the dynamics for arbitrary degrees, including those that fail the condition $\gcd(n,p^3-p)=1$ used in the prime-power permutation setting.

We first state the fixed-point formula. Write $|A|$ for the cardinality of a finite set~$A$, and, for a map $f:X\to X$, set $\Fix(f,X)=\{a\in X:f(a)=a\}$. A fixed point $a\in\Fp$ of a polynomial $f\in\Fp[x]$ is \emph{degenerate} if $f'(a)=1$, equivalently if it is a multiple root of $f(x)-x$, and \emph{nondegenerate} otherwise.\label{def:degen} For a degree $m\ge2$, put
\[
  N_{k,p}(m)=|\Fix(T_m,\Z/p^k\Z)|,\qquad
  d_p(m)=\bigl|\{a\in\Fix(T_m,\Fp):T_m'(a)=1\}\bigr|.
\]
For the fixed degree $n$, abbreviate $N_k=N_{k,p}(n)$ and $d=d_p(n)$. We write $\nu_p$ for the valuation on the field $\Q_p$ of $p$-adic numbers and its finite extensions, normalized by $\nu_p(p)=1$ and $\nu_p(0)=+\infty$. The bracket $[P]$ equals $1$ if the statement $P$ holds and $0$ otherwise.

\begin{theorem}[Fixed-point count over odd prime-power rings]
\label{thm:complete-fixed-count}
Let $p$ be an odd prime, $n\ge2$, and $k\ge1$. Set
\[
  G_{1,k}=\gcd(n-1,(p-1)p^{k-1}),\qquad
  G_{2,k}=\gcd(n+1,(p-1)p^{k-1}),
\]
\[
  G_{3,k}=\gcd(n-1,(p+1)p^{k-1}),\qquad
  G_{4,k}=\gcd(n+1,(p+1)p^{k-1}),
\]
and put $\beta=\gcd(n-1,2)$ and $s=\nup(n^2-1)$. Then
\begin{equation}\label{eq:complete-fixed-count}
  N_k=
  \frac{G_{1,k}+G_{2,k}+G_{3,k}+G_{4,k}-2\beta}{2}
  +\Delta_{p,n,k},
\end{equation}
where
\begin{equation}\label{eq:ternary-correction}
  \Delta_{p,n,k}
  =[p=3][3\nmid n][k\ge s+2]\,\beta\,3^s.
\end{equation}
Equivalently,
\begin{equation}\label{eq:complete-fixed-lifting}
  N_k=N_1+d\bigl(p^{\min(k-1,s)}-1\bigr)+\Delta_{p,n,k}.
\end{equation}
\end{theorem}

The lifting form~\eqref{eq:complete-fixed-lifting} explains the structure of the count. Every nondegenerate fixed residue has one fixed lift at each level. A degenerate nonboundary residue has $p^{\min(k-1,s)}$ fixed lifts, and the same holds at a degenerate boundary residue when $p\ge5$. Proposition~\ref{prop:degen} gives an explicit formula for the number $d$ of degenerate fixed residues. The proof of the theorem in Section~\ref{sec:lifting} combines these local counts with the separate boundary calculation at~$p=3$.

The correction at $p=3$ comes from two distinct rational fixed points with the same reduction. If $3\nmid n$ and $\varepsilon\in\{\pm1\}$ satisfies $T_n(\varepsilon)=\varepsilon$, then $-\varepsilon/2$ is also fixed. These points are congruent modulo $3$ but distinct modulo $9$. For $k\ge s+2$, Corollary~\ref{cor:ternary-boundary-count} places the fixed lifts in two disjoint congruence classes, one around each rational point, each containing $3^s$ residues modulo $3^k$. Their combined contribution is therefore $2\cdot3^s$ instead of $3^s$, giving the additional term $\Delta_{3,n,k}$ after summing over the $\beta$ fixed boundary residues.

To state the period results, we need an order that allows either sign. We write $\operatorname{ord}_G(u)$ for the order of an element $u$ of a finite group~$G$, omitting the subscript when the group is clear. For $\gcd(n,M)=1$ and $M\ge2$, the notation $\operatorname{ord}_M(n)$ refers to the multiplicative order of the residue class of $n$ in $(\Z/M\Z)^\times$.

\begin{definition}[Signed order]\label{def:signed-order}
For integers $n\ge2$ and $M\ge1$ with $\gcd(n,M)=1$, the \emph{signed order} of $n$ modulo~$M$ is
\[
  \cord_M(n)=\min\{r\ge1:n^r\equiv\pm1\pmod M\}.
\]
For $M\ge3$, this is the order of the class of $n$ in $(\Z/M\Z)^\times/\{\pm1\}$. In particular, $\cord_1(n)=1$, and $\cord_2(n)=1$ when $n$ is odd.
\end{definition}

The finite-field criterion recalled in Proposition~\ref{thm:period} states that a residue $a\in\Fp$ of source order $e$ is periodic if and only if $\gcd(n,e)=1$. In that case its period is $t=\cord_e(n)$. Suppose first that $p\nmid n$. The iterate $T_n^{\circ t}$ is the return map at~$a$, and its derivative $\lambda=(T_n^{\circ t})'(a)\in\Fp^\times$ is the \emph{return multiplier}. Proposition~\ref{prop:return-multiplier} gives
\[
  \operatorname{ord}_{\Fp^\times}(\lambda)
  =\frac{\cord_{ep}(n)}{\cord_e(n)}.
\]
Theorem~\ref{thm:first-period-lift} translates this quotient into a precise first-lift rule: the possible periods modulo $p^2$ above $a$ are $t$ and $\cord_{ep}(n)$, with their multiplicities determined by whether these two values coincide.

The higher periods are controlled by the sequence
\[
  c_q=\cord_{ep^q}(n)\quad(q\ge0),\qquad
  R=\cord_p(n),\qquad w=\nu_p(n^{2R}-1).
\]
Proposition~\ref{prop:signed-order-growth} proves that
\[
  c_0=t,\qquad c_q=c_1p^{\max(0,q-w)}\quad(q\ge1).
\]
Thus, after a possible initial increase from $c_0$ to $c_1$, the sequence is constant through $q=w$ and then grows by a factor of $p$ at each step. Every lift of $a$ modulo $p^k$ is periodic. If $p\ge5$ or $a\notin\{\pm1\}$, the periods that occur are precisely the distinct values among $c_0,\ldots,c_{k-1}$, and Theorem~\ref{thm:explicit-period-tower} gives the number of points of each period. Theorem~\ref{thm:ternary-boundary-periods} gives the distribution for the boundary residues when $p=3$; there, all periods are powers of $3$, with separate multiplicity formulas. These statements concern point counts above a single residue; the corresponding cycle counts above an entire cycle are given in Corollary~\ref{cor:cycle-count-tower} and the discussion following Theorem~\ref{thm:ternary-boundary-periods}.

When $p\mid n$, the behavior is simpler: reduction modulo~$p$ is a bijection on periodic points and preserves their periods (Proposition~\ref{prop:p-divides-n}). Thus every periodic residue has exactly one periodic lift at each level. For all degrees, the identity $T_n^{\circ j}=T_{n^j}$ also allows us to apply Theorem~\ref{thm:complete-fixed-count} to every iterate. M\"obius inversion then gives the number of points of any prescribed exact period modulo $p^k$ (Corollary~\ref{cor:Prk-mobius}).

The proofs use the different local behavior of $F$ at boundary and nonboundary parameters. At a nonboundary parameter its derivative is a unit, and Hensel's lemma converts fixed-point lifting modulo $p^k$ into a power equation in a cyclic group of order $p^{k-1}$. At the boundary parameters the derivative vanishes, and the count instead follows from polynomial factorizations and $p$-adic valuation estimates. For $p=3$, the second rational fixed point supplies the additional factor needed for the valuation argument. For $p\nmid n$, the period sequences just described contain all possible periods above each periodic residue and form divisibility chains. Their fixed sets are therefore nested; after repeated periods are removed, successive differences give the exact-period multiplicities.

In the permutation cases, the cycle counts recover those of Li--Lu--Tan--Chen~\cite[Proposition~7]{LLTC2025} for $p>3$ and Lu--Dai--Li~\cite[Propositions~4 and~5]{LuDaiLi2026} for $p=3$. The argument here applies to arbitrary degrees and accounts for the exceptional boundary counts through the two rational fixed points. Precise parameter comparisons follow Corollary~\ref{cor:cycle-count-tower} and Theorem~\ref{thm:ternary-boundary-periods}.

Section~\ref{sec:prelim} records the parameter identities and finite-field counts. Section~\ref{sec:lifting} proves Theorem~\ref{thm:complete-fixed-count}. Section~\ref{sec:period-Fp} relates signed orders to return multipliers and determines the first lifts. Section~\ref{sec:higher-level} gives the higher period multiplicities and global exact-period counts.

\section{Parameters and fixed points over finite fields}\label{sec:prelim}

\subsection{Identities and notation}
We first record the identities needed for the parameter calculations. In addition to the first-kind polynomials introduced in Section~\ref{sec:intro}, we use the Chebyshev polynomials of the second kind, defined by
\[
 U_0=1,\qquad U_1=2x,\qquad U_{m+1}=2xU_m-U_{m-1}\quad(m\ge1).
\]
The first-kind polynomials satisfy $T_0=1$, $T_1=x$, and $T_{m+1}=2xT_m-T_{m-1}$ for $m\ge1$. Both families belong to $\Z[x]$, and
\begin{equation}\label{eq:TU-def}
 U_m(\cos\theta)\sin\theta=\sin((m+1)\theta)\qquad(m\ge0).
\end{equation}
The standard identities
\[
 T_m'=mU_{m-1},\qquad
 T_m(x)^2-1=(x^2-1)U_{m-1}(x)^2\qquad(m\ge1)
\]
and the values
\[
 T_m(1)=1,\quad T_m(-1)=(-1)^m,\quad
 U_{m-1}(1)=m,\quad U_{m-1}(-1)=(-1)^{m-1}m
\]
will be used in the lifting arguments; see~\cite{Rivlin1990,MH2003}. We also use the parity identity $T_m(-x)=(-1)^mT_m(x)$. For the Dickson normalization mentioned in the introduction, see~\cite{Dickson1896,LMT1993,WangYucas2012}.

Induction on the recurrences gives the parameter identities
\begin{equation}\label{eq:param}
\begin{aligned}
 T_m\!\left(\frac{z+z^{-1}}2\right)&=\frac{z^m+z^{-m}}2,\\
 U_{m-1}\!\left(\frac{z+z^{-1}}2\right)(z-z^{-1})&=z^m-z^{-m}
 \qquad(m\ge1).
\end{aligned}
\end{equation}
These hold in the Laurent polynomial ring $\Z[1/2][z,z^{-1}]$, so they may be evaluated at any unit of a commutative ring in which $2$ is invertible. The second identity is written without division by $z-z^{-1}$ so that it remains valid at the boundary parameters as well.

We retain the notation of Section~\ref{sec:intro}. For a periodic residue $a\in\Z/p^k\Z$, its \emph{exact period} $\operatorname{per}_{p^k}(a)$ is the least $r\ge1$ such that $T_n^{\circ r}(a)=a$. We write $P_r^{(k)}$ for the number of points of exact period~$r$, so $P_r^{(k)}/r$ is the number of cycles of length~$r$. Euler's totient function is denoted by $\varphi$, and the M\"obius function by $\Mob$. The normalized valuation $\nu_p$ may take fractional values in ramified extensions of~$\Q_p$.

We shall also use the following form of Hensel's lemma~\cite[Ch.~I, Sec.~6]{Koblitz1984}.
\begin{lemma}[Hensel's lemma]\label{thm:hensel}
Let $R$ be a complete discrete valuation ring with maximal ideal $\mathfrak m$, and let $g\in R[x]$. Suppose $\bar a\in R/\mathfrak m$ is a simple root of the reduction $\bar g$, that is, $\bar g(\bar a)=0$ and $\bar g'(\bar a)\ne0$. Then there is exactly one root of $g$ in $R$ reducing to $\bar a$. For every $k\ge1$, there is also exactly one root modulo $\mathfrak m^k$ reducing to~$\bar a$.
\end{lemma}
Our applications use $R=\Z_p$ or the ring of integers of the unramified quadratic extension of~$\Q_p$. In particular, applying the lemma to $T_n(x)-x$ shows that each nondegenerate fixed residue has exactly one fixed lift modulo every~$p^k$.

\subsection{The two-source parametrization}\label{sec:dynamics}
We now justify the parametrization introduced in Section~\ref{sec:intro} and identify its fibers. This description is standard in the finite-field theory of Chebyshev maps; see~\cite{Gassert2014,QP2019,RosenScherrWeissZieve2012}.

For $a\in\Fp$, set
\[
 H_a(Z)=Z^2-2aZ+1.
\]

\begin{lemma}[Parameter identities over $\Fp$]\label{lem:zeta}
The roots of $H_a$ lie in $\Fpp^\times$ and form an inverse pair, with repetition when $a=\pm1$. For either root $\zeta$, one has $a=F(\zeta)$ and
\[
 T_n(a)=\frac{\zeta^n+\zeta^{-n}}2.
\]
If $a\ne\pm1$, then $\zeta\ne\pm1$, the two roots are distinct, and
\[
 U_{n-1}(a)=\frac{\zeta^n-\zeta^{-n}}{\zeta-\zeta^{-1}}.
\]
\end{lemma}

\begin{proof}
A quadratic over $\Fp$ splits over $\Fpp$. The constant term of $H_a$ shows that its roots are nonzero and have product~$1$, while their sum is~$2a$. Its discriminant is $4(a^2-1)$, and $H_a(\pm1)=2\mp2a$. Thus the roots are distinct and different from $\pm1$ when $a\ne\pm1$. The formulas now follow from~\eqref{eq:param}; the denominator in the last formula is nonzero under this hypothesis.
\end{proof}

\begin{proposition}[Two-source parametrization]\label{prop:two-source}
The source groups $\Fp^\times$ and $\mu_{p+1}$ are cyclic of orders $p-1$ and $p+1$, respectively, and their intersection is $\{\pm1\}$. For $\zeta\in\Fpp^\times$,
\[
 F(\zeta)\in\Fp
 \quad\Longleftrightarrow\quad
 \zeta\in\Fp^\times\cup\mu_{p+1}.
\]
The restriction of $F$ to this union is surjective onto~$\Fp$, and its fibers are exactly the orbits under inversion. The fibers over $1$ and $-1$ are the singletons $\{1\}$ and $\{-1\}$; every other fiber is a pair $\{\zeta,\zeta^{-1}\}$ contained in exactly one source group. For $a\ne\pm1$, this group is $\Fp^\times$ when $H_a$ splits over $\Fp$, and $\mu_{p+1}$ when $H_a$ is irreducible over~$\Fp$.
\end{proposition}

\begin{proof}
The group $\Fpp^\times$ is cyclic of order $p^2-1$, so its subgroups $\Fp^\times$ and $\mu_{p+1}$ are cyclic of the stated orders. The latter is the kernel of the norm map $\zeta\mapsto\zeta^{p+1}$. For $\zeta\in\Fp^\times$, the condition $\zeta^{p+1}=1$ is equivalent to $\zeta^2=1$, which proves the intersection statement.

If $\zeta$ belongs to either source group, then $\zeta^p=\zeta$ or $\zeta^p=\zeta^{-1}$, respectively, and hence $F(\zeta)^p=F(\zeta)$. Conversely, if $a=F(\zeta)\in\Fp$, Frobenius permutes the roots $\zeta,\zeta^{-1}$ of $H_a$. Therefore $\zeta^p=\zeta$ or $\zeta^p=\zeta^{-1}$, placing $\zeta$ in one of the two source groups. Lemma~\ref{lem:zeta} supplies such a parameter for every $a\in\Fp$, proving surjectivity.

For $\zeta,\xi\in\Fpp^\times$, the equality $F(\zeta)=F(\xi)$ is equivalent to
\[
 (\zeta-\xi)(1-\zeta^{-1}\xi^{-1})=0.
\]
Since $\Fpp$ is a field, this holds precisely when $\xi=\zeta$ or $\xi=\zeta^{-1}$. These coincide only at $\zeta=\pm1$. Outside these two parameters the source groups are disjoint, and each is closed under inversion. Finally, the roots of $H_a$ lie in $\Fp$ exactly when $H_a$ splits there; otherwise the preceding dichotomy places them in $\mu_{p+1}\setminus\Fp^\times$.
\end{proof}

\begin{definition}[Split and nonsplit points]\label{def:split}
A point $a\in\Fp\setminus\{\pm1\}$ is \emph{split} if $H_a$ splits over $\Fp$, and \emph{nonsplit} if $H_a$ is irreducible over $\Fp$. Equivalently, its source parameters lie in $\Fp^\times\setminus\{\pm1\}$ or in $\mu_{p+1}\setminus\{\pm1\}$, respectively.
\end{definition}

To count fixed points, we next translate $T_n(a)=a$ into two power equations in the source groups.

\begin{lemma}[Fixed-point factorization]\label{lem:fixed-factor}
For $a\in\Fp$ with source parameter $\zeta$,
\[
 T_n(a)=a
 \quad\Longleftrightarrow\quad
 \zeta^{n-1}=1\ \text{or}\ \zeta^{n+1}=1.
\]
If $a\ne\pm1$ is fixed, exactly one of the two power equations holds.
\end{lemma}

\begin{proof}
Identity~\eqref{eq:param} gives
\[
 T_n(F(\zeta))-F(\zeta)
 =\frac{(\zeta^{n-1}-1)(\zeta^{n+1}-1)}{2\zeta^n}.
\]
The denominator is nonzero, and a product in $\Fpp$ vanishes if and only if one of its factors vanishes. If both power equations hold, division gives $\zeta^2=1$, which forces $a=\pm1$.
\end{proof}

\begin{remark}
The fixed boundary residues are $1$ and, when $n$ is odd, $-1$. Thus their number is $\beta=\gcd(n-1,2)$, as in Theorem~\ref{thm:complete-fixed-count}. Their parameters satisfy both power equations in Lemma~\ref{lem:fixed-factor}; these are the only overlaps between the two equations in either source group.
\end{remark}

\subsection{Finite-field fixed points}\label{sec:N1}
The following formula is the fixed-point specialization of the known finite-field graph descriptions~\cite{Gassert2014,QP2019,HutzPatel2022}. We give a direct proof to account for the overlaps before turning to lifting.

\begin{proposition}\label{thm:N1}
For an odd prime $p$ and $n\ge2$,
\begin{equation}\label{eq:N1}
 N_1=\frac{G_1+G_2+G_3+G_4-2\beta}{2},
\end{equation}
where $G_i=G_{i,1}$ for $1\le i\le4$, with the notation of Theorem~\ref{thm:complete-fixed-count}.
\end{proposition}

\begin{proof}
Let $S$ be the set of parameters in $\Fp^\times\cup\mu_{p+1}$ satisfying $\zeta^{n-1}=1$ or $\zeta^{n+1}=1$. By Lemma~\ref{lem:fixed-factor} and Proposition~\ref{prop:two-source}, the fixed points of $T_n$ correspond bijectively to the inversion orbits on~$S$.

In a cyclic group of order $M$, the equation $z^j=1$ has $\gcd(j,M)$ solutions. The two power equations therefore have $G_1$ and $G_2$ solutions in $\Fp^\times$, and $G_3$ and $G_4$ solutions in $\mu_{p+1}$. Within each group, their intersection consists of the $\beta$ fixed boundary parameters. Hence
\[
 |S\cap\Fp^\times|=G_1+G_2-\beta,\qquad
 |S\cap\mu_{p+1}|=G_3+G_4-\beta.
\]
These two sets also intersect in exactly those $\beta$ parameters, giving
\[
 |S|=G_1+G_2+G_3+G_4-3\beta.
\]
Inversion fixes the $\beta$ boundary parameters and pairs the remaining elements. It follows that
\[
 N_1=\beta+\frac{|S|-\beta}{2}
     =\frac{G_1+G_2+G_3+G_4-2\beta}{2}.\qedhere
\]
\end{proof}

\begin{example}\label{ex:N1}
For $p=7$ and $n=5$, one has $(G_1,G_2,G_3,G_4)=(2,6,4,2)$ and $\beta=2$, so $N_1=5$. Directly,
\[
 T_5(x)-x\equiv2x(x^2-1)(x^2-2)\pmod7,
\]
whose roots in $\mathbb F_7$ are exactly $0,1,3,4,6$.
\end{example}

\begin{corollary}[Chebyshev--Fermat]\label{cor:fermat}
For every odd prime $p$, the map $T_p$ fixes every element of $\Fp$. In particular, $N_1=p$ when $n=p$.
\end{corollary}

\begin{proof}
Write $a=F(\zeta)$ with $\zeta$ in a source group. Then $\zeta^p$ is either $\zeta$ or $\zeta^{-1}$, so~\eqref{eq:param} gives $T_p(a)=F(\zeta^p)=F(\zeta)=a$.
\end{proof}

For related $p$-adic properties of $T_p$, see Diarra and Sylla~\cite{DiarraSylla2014}.

\begin{remark}[Rational fixed points and reduction modulo $3$]\label{rem:universal-fixed}
Some fixed points can be identified over $\Q$ before reduction. If $\zeta$ is a complex root of unity of order $h$, the factorization used in Lemma~\ref{lem:fixed-factor} gives
\[
 T_n(F(\zeta))=F(\zeta)
 \quad\Longleftrightarrow\quad n\equiv\pm1\pmod h.
\]
For $h\in\{1,2,3,4,6\}$, this yields the following rational values and conditions:
\begingroup
\renewcommand{\arraystretch}{1.25}
\setlength{\arraycolsep}{9pt}
\[
\begin{array}{r r l}
\toprule
h & a=F(\zeta) & \text{Condition for }T_n(a)=a\text{ over }\Q\\
\midrule
1 & 1 & \text{all }n\\
2 & -1 & n\text{ odd}\\
4 & 0 & n\text{ odd}\\
3 & -1/2 & 3\nmid n\\
6 & 1/2 & n\equiv\pm1\pmod6\\
\bottomrule
\end{array}
\]
\endgroup
Here $h$ is an order in characteristic zero, and need not equal the source order after reduction. Under the respective conditions in the table, these rational points are fixed modulo $p^k$ for every odd prime $p$ and every $k\ge1$, since their denominators are invertible. All five are fixed whenever $n\equiv\pm1\pmod6$. In particular, this holds under the permutation condition $\gcd(n,p^3-p)=1$, because $6\mid p^3-p$. For $p\ge5$ the five residues are distinct; in the permutation setting they are the fixed points listed in~\cite[Proposition~4]{LLTC2025}.

For $p=3$, the only possible source orders are $1,2,4$, and the five rational points reduce to just three residues. If $3\nmid n$, both $1$ and $-1/2$ are fixed and reduce to~$1$; if $n$ is also odd, both $-1$ and $1/2$ are fixed and reduce to~$-1$. Thus each fixed boundary residue contains two distinct rational fixed points when $3\nmid n$. Subsection~\ref{subsec:ternary} computes their contribution to the fixed-point count at higher powers of~$3$.
\end{remark}

\section{Fixed-point lifting}\label{sec:lifting}

We prove Theorem~\ref{thm:complete-fixed-count} by counting fixed lifts above each fixed residue modulo~$p$. For $a_0\in\Fp$, its \emph{residue disk} is the set of $p$-adic integers reducing to~$a_0$; the classes in this disk modulo $p^k$ are precisely the lifts of~$a_0$. Hensel's lemma handles the nondegenerate residues. For degenerate residues, we use the parameter map away from $\pm1$ and polynomial factorizations at the boundary. The latter require a separate calculation when $p=3$.

\subsection{Derivatives and the first lift}
We begin by determining which fixed residues are degenerate. The source parameter identifies the sign in the derivative formula $T_n'(a_0)=\pm n$ away from the boundary, refining the reduction modulo $p$ of~\cite[Lemma~4]{LLTC2025}.

\begin{proposition}[Degeneracy criterion]\label{prop:degen}
Let $a_0\in\Fp$ be fixed by $T_n$.
\begin{enumerate}[label=\textup{(\alph*)},nosep]
\item If $a_0\ne\pm1$, choose a source parameter $\zeta_0$. There is a unique sign $\sigma\in\{\pm1\}$ such that $\zeta_0^{n-\sigma}=1$, and
\[
 T_n'(a_0)=\sigma n\quad\text{in }\Fp.
\]
Thus $a_0$ is degenerate if and only if $p\mid n-\sigma$.
\item At a fixed boundary point $\varepsilon\in\{\pm1\}$, one has $T_n'(\varepsilon)=n^2$. Thus $\varepsilon$ is degenerate if and only if $p\mid n^2-1$.
\end{enumerate}
In particular, degenerate fixed residues exist if and only if $p\mid n^2-1$. Their number is
\begin{equation}\label{eq:d-count}
 d=
 \begin{cases}
  (G_1+G_3)/2,&p\mid n-1,\\[1mm]
  (G_2+G_4)/2,&p\mid n+1,\\[1mm]
  0,&p\nmid n^2-1,
 \end{cases}
\end{equation}
where $G_i=G_{i,1}$ as in Proposition~\ref{thm:N1}.
\end{proposition}

\begin{proof}
For a nonboundary fixed point, the existence and uniqueness of $\sigma$ follow from Lemma~\ref{lem:fixed-factor}. Since $\zeta_0^n=\zeta_0^\sigma$, Lemma~\ref{lem:zeta} gives
\[
 U_{n-1}(a_0)
 =\frac{\zeta_0^n-\zeta_0^{-n}}{\zeta_0-\zeta_0^{-1}}
 =\sigma.
\]
The identity $T_n'=nU_{n-1}$ proves~(a). At the boundary, the values of $U_{n-1}$ in Section~\ref{sec:prelim} give $T_n'(1)=n^2$ and $T_n'(-1)=(-1)^{n-1}n^2$. The point $-1$ is fixed only when $n$ is odd, proving~(b). These formulas show that degeneracy requires $p\mid n^2-1$; conversely, this divisibility makes the fixed point~$1$ degenerate.

Suppose $p\mid n-\sigma$ for a sign $\sigma\in\{\pm1\}$. Since $p$ is odd, this sign is unique. Parts~(a) and~(b) identify the degenerate fixed residues with the images under $F$ of
\[
 S_\sigma=\{\zeta\in\Fp^\times\cup\mu_{p+1}:\zeta^{n-\sigma}=1\}.
\]
Indeed, this equation selects exactly the degenerate nonboundary fixed points and includes every fixed boundary parameter. It has $\gcd(n-\sigma,p-1)$ solutions in the first source group and $\gcd(n-\sigma,p+1)$ in the second. The two sets intersect in the $\beta$ fixed boundary parameters, which are also exactly the elements of $S_\sigma$ fixed by inversion. Proposition~\ref{prop:two-source} therefore gives
\[
 d=\frac{|S_\sigma|+\beta}{2}
  =\frac{\gcd(n-\sigma,p-1)+\gcd(n-\sigma,p+1)}2.
\]
This is the first or second case of~\eqref{eq:d-count}; the remaining case follows from the absence of degenerate fixed residues.
\end{proof}

For a nonboundary fixed point, we call the unique sign $\sigma$ in Proposition~\ref{prop:degen}(a) its \emph{branch sign}. It is unchanged on replacing $\zeta_0$ by $\zeta_0^{-1}$. In the degenerate case,
\[
 s=\nu_p(n^2-1)=\nu_p(n-\sigma),\qquad p\nmid n+\sigma.
\]

We next construct a $p$-adic fixed lift. Throughout this section, $K/\Q_p$ denotes the unramified quadratic extension, $\mathcal O_K$ its ring of integers, and $\tau$ its nontrivial automorphism. If $R=\Z_p$ or $\mathcal O_K$ has residue field of size $q$, the \emph{Teichm\"uller lift} of a nonzero residue $\zeta_0$ is the unique root $\tilde\zeta\in R$ of $Z^{q-1}-1$ reducing to~$\zeta_0$. Hensel's lemma gives existence and uniqueness. If $\zeta_0$ has order $e$, applying the same lemma to $Z^e-1$ shows that $\tilde\zeta$ also has order~$e$.

The conclusion about degenerate lifts modulo $p^2$ agrees with the corresponding case of~\cite[Proposition~1]{LLTC2025}, proved there for permutation degrees and $p>3$. We give a direct construction for arbitrary degrees and state it with degree $m$ for later use with iterates.

\begin{lemma}[Local first lift]\label{lem:local-first-lift}
Let $m\ge2$, and let $a_0\in\Fp$ be fixed by $T_m$. There is a fixed point of $T_m$ in $\Z_p$ reducing to~$a_0$. Among the $p$ lifts of $a_0$ modulo $p^2$, exactly one is fixed if $a_0$ is nondegenerate for $T_m$, and all are fixed if $a_0$ is degenerate.
\end{lemma}

\begin{proof}
If $a_0=\varepsilon\in\{\pm1\}$, take $\alpha=\varepsilon$; the point $-1$ is fixed modulo an odd prime only when $m$ is odd. Otherwise choose a source parameter $\zeta_0$ and its branch sign $\sigma$. Take its Teichm\"uller lift $\tilde\zeta$ in $\Z_p$ in the split case and in $\mathcal O_K$ in the nonsplit case. Preservation of order gives $\tilde\zeta^{m-\sigma}=1$.

Set $\alpha=F(\tilde\zeta)$. In the nonsplit case, $\tau(\tilde\zeta)$ and $\tilde\zeta^{-1}$ are Teichm\"uller lifts of the same residue, since $\zeta_0^p=\zeta_0^{-1}$. They are therefore equal, so $\alpha$ is fixed by $\tau$ and belongs to~$\Z_p$. This is immediate in the split case. Identity~\eqref{eq:param} now gives $T_m(\alpha)=\alpha$.

If $a_0$ is nondegenerate, uniqueness of the fixed lift modulo $p^2$ follows from Hensel's lemma. If it is degenerate, every lift is represented by $\alpha+pu$ with $u\in\Z_p$, and Taylor expansion gives
\[
 T_m(\alpha+pu)-(\alpha+pu)
 \equiv pu\bigl(T_m'(\alpha)-1\bigr)\equiv0\pmod{p^2}.
\]
Thus all $p$ lifts are fixed.
\end{proof}

\begin{corollary}[First-lift count]\label{thm:first-lift}
For every odd prime $p$ and $n\ge2$,
\[
 N_2=N_1+d(p-1).
\]
\end{corollary}

\begin{proof}
Lemma~\ref{lem:local-first-lift} gives one fixed lift for each of the $N_1-d$ nondegenerate residues and $p$ for each of the $d$ degenerate residues. Hence $N_2=(N_1-d)+dp$.
\end{proof}

For example, the split and nonsplit cases both occur in the following counts:
\[
\begin{array}{c c c c c c}
\toprule
(p,n)&\text{source group of }a_0&a_0&N_1&d&N_2\\
\midrule
(7,22)&\Fp^\times&3&2&2&14\\
(5,14)&\mu_{p+1}&2&2&2&10\\
\bottomrule
\end{array}
\]
In the first row, the parameters of $a_0=3$ are $2,4\in\mathbb F_7^\times$, both of order~$3$. In the second, $H_2(Z)=Z^2+Z+1$ over $\mathbb F_5$ is irreducible, and its roots have order~$3$ in $\mu_6$. Proposition~\ref{prop:degen} and Corollary~\ref{thm:first-lift} give the displayed values of $d$ and~$N_2$.

\subsection{Nonboundary degenerate classes}\label{sec:all-level}
For the higher lifts, we vary a Teichm\"uller parameter by units congruent to $1$ modulo~$p$, called \emph{principal units}. In the nonsplit case these units must have norm one. The next two lemmas show that, in either case, the lifts of a nonboundary residue are parametrized by a cyclic group of order~$p^{k-1}$.

\begin{lemma}[Norm-one principal units]\label{lem:nonsplit-kernel}
Write $N=N_{K/\Q_p}$. For $k\ge1$, the norm induces a homomorphism
\[
 \overline N_k:\frac{1+p\mathcal O_K}{1+p^k\mathcal O_K}
 \longrightarrow\frac{1+p\Z_p}{1+p^k\Z_p},
 \qquad[\eta]\longmapsto[N(\eta)].
\]
Its kernel is cyclic of order $p^{k-1}$. Every class in this kernel has a representative of norm exactly~$1$, and reduction induces an isomorphism
\[
 \frac{(1+p\mathcal O_K)\cap\ker N}
      {(1+p^k\mathcal O_K)\cap\ker N}
 \ \cong\ \ker\overline N_k.
\]
\end{lemma}

\begin{proof}
Since $p$ is odd and $K/\Q_p$ is unramified, logarithm and exponential give inverse group isomorphisms between $1+p^j\mathcal O_K$ and $p^j\mathcal O_K$ for every $j\ge1$; see~\cite[Ch.~IV]{Koblitz1984}. Moreover,
\[
 \log N(\eta)=\operatorname{Tr}_{K/\Q_p}(\log\eta).
\]
Put $V=\ker(\operatorname{Tr}:\mathcal O_K\to\Z_p)$. Because $\operatorname{Tr}(1)=2$ is a unit, the decomposition
\[
 x=\tfrac12\operatorname{Tr}(x)
   +\bigl(x-\tfrac12\operatorname{Tr}(x)\bigr)
\]
gives $\mathcal O_K=\Z_p\oplus V$, with $V$ free of rank one over~$\Z_p$. Thus the kernel of trace on $p\mathcal O_K/p^k\mathcal O_K$ is
\[
 pV/p^kV\cong\Z/p^{k-1}\Z.
\]
Every class in this kernel is represented by an element of $pV$, whose exponential has norm exactly~$1$. Under logarithm, the exact norm-one subgroups at levels $1$ and $k$ correspond to $pV$ and $p^kV$, respectively. This proves all the assertions, including the case $k=1$.
\end{proof}

\begin{lemma}[Hensel inverse for the averaging map]\label{lem:hensel-average}
Let $a_0\in\Fp\setminus\{\pm1\}$, choose a source parameter $\zeta_0$, and let $\tilde\zeta$ be its Teichm\"uller lift. Set
\[
 (R,\mathcal U)=
 \begin{cases}
  (\Z_p,1+p\Z_p),&a_0\text{ split},\\
  (\mathcal O_K,(1+p\mathcal O_K)\cap\ker N),&a_0\text{ nonsplit},
 \end{cases}
\]
where $N=N_{K/\Q_p}$ in the nonsplit case. For $k\ge1$, put
\[
 \mathcal P_k=\mathcal U/\bigl(\mathcal U\cap(1+p^kR)\bigr).
\]
Then $\mathcal P_k$ is cyclic of order $p^{k-1}$, and the map
\[
 [\eta]\longmapsto F(\tilde\zeta\eta)\pmod{p^k}
\]
is a well-defined bijection from $\mathcal P_k$ onto the lifts of $a_0$ modulo~$p^k$.
\end{lemma}

\begin{proof}
In the split case, logarithm identifies $\mathcal P_k$ with $p\Z_p/p^k\Z_p$. In the nonsplit case, its cyclicity and order follow from Lemma~\ref{lem:nonsplit-kernel}.

For $\eta\in\mathcal U$, put $z=\tilde\zeta\eta$. In the split case, $F(z)\in\Z_p$ is immediate. In the nonsplit case, the Teichm\"uller construction gives $\tau(\tilde\zeta)=\tilde\zeta^{-1}$, so $N(\tilde\zeta)=1$. Since $N(\eta)=1$, we have $\tau(z)=z^{-1}$ and again $F(z)\in\Z_p$. Congruent units modulo $p^kR$ have congruent images under $F$, and $p^kR\cap\Z_p=p^k\Z_p$, proving that the map is well defined.

Given a lift of $a_0$ modulo $p^k$, choose a representative $\hat x\in\Z_p$. The polynomial $Z^2-2\hat xZ+1$ has the simple root $\zeta_0$ modulo $p$, with derivative $\zeta_0-\zeta_0^{-1}\ne0$. Hensel's lemma supplies a root $z\in R$ reducing to~$\zeta_0$, and $F(z)=\hat x$. In the nonsplit case, $\tau(z)$ and $z^{-1}$ are roots reducing to $\zeta_0^{-1}$; uniqueness in Hensel's lemma gives $\tau(z)=z^{-1}$. Thus $z/\tilde\zeta\in\mathcal U$ in both cases, proving surjectivity.

Finally, for parameters $z_1,z_2$ reducing to the chosen $\zeta_0$,
\[
 F(z_1)-F(z_2)
 =\frac{z_1-z_2}{2}\left(1-\frac1{z_1z_2}\right).
\]
The factor in parentheses reduces to $1-\zeta_0^{-2}\ne0$ and is a unit. Therefore $F(z_1)\equiv F(z_2)\pmod{p^k}$ if and only if $z_1\equiv z_2\pmod{p^kR}$. This proves injectivity. Choosing one of the two distinct parameters modulo $p$ also removes the global ambiguity under inversion.
\end{proof}

\begin{lemma}[Local count above a nonboundary degenerate class]\label{lem:local-nonboundary}
Let $a_0\in\Fp\setminus\{\pm1\}$ be a degenerate fixed point of $T_n$, and set $s=\nu_p(n^2-1)$. For every $k\ge1$, the number of fixed lifts of $a_0$ modulo $p^k$ is
\[
 p^{\min(k-1,s)}.
\]
\end{lemma}

\begin{proof}
Choose a source parameter $\zeta_0$ with branch sign $\sigma$, and use $R$, $\mathcal U$, and $\mathcal P_k$ from Lemma~\ref{lem:hensel-average}. Since the Teichm\"uller lift preserves order, $\tilde\zeta^{n-\sigma}=1$. The lemma parametrizes the lifts by $[\eta]\in\mathcal P_k$. For a representative $\eta\in\mathcal U$, put $z=\tilde\zeta\eta$ and $x=F(z)\in\Z_p$. The factorization from Lemma~\ref{lem:fixed-factor} gives
\[
 T_n(x)-x
 =\frac{(z^{n-1}-1)(z^{n+1}-1)}{2z^n}.
\]
The factor $z^{n+\sigma}-1$ is a unit: otherwise both power equations would hold at $\zeta_0$ modulo $p$, contrary to the uniqueness of the branch sign. Since $p^kR\cap\Z_p=p^k\Z_p$, we obtain
\[
 \begin{aligned}
 T_n(x)\equiv x\pmod{p^k}
 &\quad\Longleftrightarrow\quad z^{n-\sigma}\equiv1\pmod{p^kR}\\
 &\quad\Longleftrightarrow\quad[\eta]^{n-\sigma}=1
       \quad\text{in }\mathcal P_k.
 \end{aligned}
\]
The unit factor is essential here, since a vanishing product modulo $p^k$ alone need not have a vanishing factor. The cyclic group $\mathcal P_k$ has order $p^{k-1}$, so the last equation has
\[
 \gcd(n-\sigma,p^{k-1})
 =p^{\min(k-1,\nu_p(n-\sigma))}
 =p^{\min(k-1,s)}
\]
solutions.
\end{proof}

\subsection{\texorpdfstring{Boundary classes for $p\ge5$}{Boundary classes for p at least 5}}
At a boundary parameter, the derivative of $F$ vanishes, so the preceding inverse argument does not apply. For a fixed boundary point $\varepsilon\in\{\pm1\}$, define the polynomial
\[
 Q_\varepsilon(x)=\frac{T_n(x)-x}{x-\varepsilon}\in\Z[x].
\]
We shall show that its valuation is constant on $\varepsilon+p\Z_p$ when $p\ge5$. The parameter calculation may take place in a ramified quadratic extension, so we first record the precise valuation range needed for logarithm and exponential.

\begin{lemma}[Logarithm and geometric-sum valuations]\label{lem:Am-valuation}
Let $E/\Q_p$ be a finite extension. If $\xi,\mu\in E$ satisfy
\[
 \nu_p(\xi)>\frac1{p-1},\qquad
 \nu_p(\mu)>\frac1{p-1},
\]
then
\[
 \nu_p(\log(1+\xi))=\nu_p(\xi),\qquad
 \nu_p(\exp(\mu)-1)=\nu_p(\mu).
\]
Consequently, for
\[
 A_m(z)=1+z+\cdots+z^{m-1}\qquad(m\ge1),
\]
one has $\nu_p(A_m(z))=\nu_p(m)$ whenever $z\in E$ satisfies $\nu_p(z-1)>1/(p-1)$.
\end{lemma}

\begin{proof}
The logarithm and exponential are inverse on the stated range; see~\cite[Ch.~IV, Sec.~1]{Koblitz1984}. To verify the valuation equalities, use Legendre's bound
\[
 \nu_p(r)\le\nu_p(r!)\le\frac{r-1}{p-1}\qquad(r\ge2).
\]
For nonzero $\xi$ and $\mu$ satisfying the hypotheses, this gives
\[
 \begin{aligned}
 \nu_p(\xi^r/r)-\nu_p(\xi)
   &=(r-1)\nu_p(\xi)-\nu_p(r)>0,\\
 \nu_p(\mu^r/r!)-\nu_p(\mu)
   &=(r-1)\nu_p(\mu)-\nu_p(r!)>0.
 \end{aligned}
\]
Thus the linear term has strictly smaller valuation than every later term in each series, proving both equalities. The zero cases follow from the convention $\nu_p(0)=+\infty$.

If $z=1$, then $A_m(z)=m$. Otherwise, logarithm and exponential give
\[
 A_m(z)=\frac{\exp(m\log z)-1}{z-1}.
\]
Since $\nu_p(m\log z)\ge\nu_p(\log z)>1/(p-1)$, the two valuation equalities yield $\nu_p(A_m(z))=\nu_p(m)$.
\end{proof}

In particular, for odd $p$ the lemma gives the lifting-the-exponent identity
\begin{equation}\label{eq:LTE}
 \nu_p(v^j-1)=\nu_p(v-1)+\nu_p(j)
 \qquad(v\in1+p\Z_p,\ v\ne1,\ j\ge1),
\end{equation}
by factoring $v^j-1=(v-1)A_j(v)$. This form will also be used for iterates.

\begin{remark}[Why the strict inequality matters]\label{rem:threshold}
Over $\Q_p$, positive valuations are integers, so they exceed $1/(p-1)$ for odd~$p$. In a ramified extension, however, a positive valuation may be fractional. The boundary calculation below gives only $\nu_p(z-1)\ge1/2$. This suffices for Lemma~\ref{lem:Am-valuation} when $p\ge5$, but reaches its threshold when $p=3$.
\end{remark}

\begin{proposition}[Valuation on a boundary residue disk]\label{prop:Qeps-val}
Let $p\ge5$, let $\varepsilon\in\{\pm1\}$ be a degenerate fixed point of $T_n$ modulo $p$, and put $s=\nu_p(n^2-1)$. Then
\[
 \nu_p(Q_\varepsilon(x))=s
 \qquad(x\in\varepsilon+p\Z_p).
\]
\end{proposition}

\begin{proof}
Fix $x\in\varepsilon+p\Z_p$, put $u=\varepsilon x\in1+p\Z_p$, and choose a root $z$ of $Z^2-2uZ+1$ in a finite extension of~$\Q_p$. The roots are integral and have product~$1$, so they are units. Moreover,
\[
 (z-1)^2=2z(u-1),\qquad \nu_p(z-1)\ge\tfrac12.
\]
Since $T_n(\varepsilon)=\varepsilon$, the parity identity gives $T_n(\varepsilon y)=\varepsilon T_n(y)$. Using~\eqref{eq:param}, we obtain
\[
 T_n(x)-x=\varepsilon\frac{(z^{n-1}-1)(z^{n+1}-1)}{2z^n},
 \qquad
 x-\varepsilon=\varepsilon\frac{(z-1)^2}{2z}.
\]
Cancelling $(z-1)^2$ as a Laurent polynomial factor gives
\begin{equation}\label{eq:Qeps-factor}
 Q_\varepsilon(x)=z^{1-n}A_{n-1}(z)A_{n+1}(z),
 \qquad x=\varepsilon F(z).
\end{equation}
This is an identity after substitution in the Laurent polynomial ring, so it also holds at $z=1$, where both sides equal $n^2-1$.

For $z\ne1$, we have $\nu_p(z-1)\ge1/2>1/(p-1)$. Applying Lemma~\ref{lem:Am-valuation} to both geometric sums yields
\[
 \nu_p(Q_\varepsilon(x))
 =\nu_p(n-1)+\nu_p(n+1)=s.
\]
The valuation on the extension restricts to the prescribed valuation on $\Q_p$, so this computes the valuation of $Q_\varepsilon(x)$ in~$\Q_p$ as required.
\end{proof}

\begin{lemma}[Local count above a boundary degenerate class]\label{lem:local-boundary}
Let $p\ge5$, let $\varepsilon\in\{\pm1\}$ be a degenerate fixed point of $T_n$ modulo $p$, and put $s=\nu_p(n^2-1)$. For every $k\ge1$, the number of fixed lifts of $\varepsilon$ modulo $p^k$ is
\[
 p^{\min(k-1,s)}.
\]
\end{lemma}

\begin{proof}
The factorization $T_n(x)-x=(x-\varepsilon)Q_\varepsilon(x)$ and Proposition~\ref{prop:Qeps-val} give
\[
 \nu_p(T_n(x)-x)=\nu_p(x-\varepsilon)+s
 \qquad(x\in\varepsilon+p\Z_p),
\]
including $x=\varepsilon$ under the convention $\nu_p(0)=+\infty$. Thus a lift is fixed modulo $p^k$ if and only if
\[
 x\equiv\varepsilon\pmod{p^{\max(1,k-s)}}.
\]
There are $p^{k-\max(1,k-s)}=p^{\min(k-1,s)}$ such classes modulo~$p^k$.
\end{proof}

\subsection{Boundary classes for \texorpdfstring{$p=3$}{p=3}}\label{subsec:ternary}
At $p=3$, the bound $\nu_3(z-1)\ge1/2$ need not lie in the range of Lemma~\ref{lem:Am-valuation}. Remark~\ref{rem:universal-fixed} identifies the additional fixed point responsible for the different count: if $3\nmid n$ and $\varepsilon$ is a fixed boundary point, then $-\varepsilon/2$ is also fixed and lies in the same residue disk. Factoring out both rational roots allows the remaining factor to be evaluated within the required range.

\begin{proposition}[Boundary factorization and valuation at $p=3$]\label{prop:ternary-boundary-valuation}
Let $n\ge2$ with $3\nmid n$, let $\varepsilon\in\{\pm1\}$ satisfy $T_n(\varepsilon)=\varepsilon$, and put $s=\nu_3(n^2-1)$. Then
\[
 R_{n,\varepsilon}(x)
 :=\frac{T_n(x)-x}{(x-\varepsilon)(2x+\varepsilon)}
 \in\Z[x],
\]
and
\[
 \nu_3(R_{n,\varepsilon}(x))=s-1
 \qquad(x\in\varepsilon+3\Z_3).
\]
\end{proposition}

\begin{proof}
By Remark~\ref{rem:universal-fixed}, both $\varepsilon$ and $-\varepsilon/2$ are rational fixed points. The primitive polynomial
\[
 (x-\varepsilon)(2x+\varepsilon)=2x^2-\varepsilon x-1
\]
therefore divides $T_n(x)-x$ in $\Q[x]$, and Gauss's lemma gives $R_{n,\varepsilon}\in\Z[x]$.

Fix $x\in\varepsilon+3\Z_3$, put $u=\varepsilon x$, and choose a root $z$ of $Z^2-2uZ+1$ in a finite extension of~$\Q_3$. The Laurent polynomial identity~\eqref{eq:Qeps-factor} holds independently of the restriction $p\ge5$. Also,
\[
 2x+\varepsilon=\varepsilon(2u+1)
   =\varepsilon z^{-1}A_3(z).
\]
Exactly one member $m$ of $\{n-1,n+1\}$ is divisible by~$3$. Write $m'$ for the other, so $s=\nu_3(m)$ and $3\nmid m'$. The identity
\[
 A_m(z)=A_3(z)A_{m/3}(z^3)
\]
allows cancellation of $A_3(z)$ in the Laurent polynomial ring, giving
\begin{equation}\label{eq:ternary-R-factor}
 R_{n,\varepsilon}(x)
 =\varepsilon z^{2-n}A_{m'}(z)A_{m/3}(z^3),
 \qquad x=\varepsilon F(z).
\end{equation}
As a Laurent polynomial identity, this also holds when $A_3(z)=0$, including at a primitive cube root of unity. No division by the value of $A_3(z)$ is needed at those parameters.

As before, $z$ is a unit and $(z-1)^2=2z(u-1)$, so $\nu_3(z-1)\ge1/2$. Consequently $z$ reduces to~$1$, and $A_{m'}(z)$ is a unit because $3\nmid m'$. Furthermore,
\[
 A_3(z)=3+3(z-1)+(z-1)^2
\]
has valuation at least~$1$. If $z^3\ne1$, it follows that
\[
 \nu_3(z^3-1)
 =\nu_3(z-1)+\nu_3(A_3(z))\ge\tfrac32>\tfrac12.
\]
Lemma~\ref{lem:Am-valuation}, now applied at $z^3$, gives
\[
 \nu_3(A_{m/3}(z^3))=\nu_3(m/3)=s-1.
\]
If $z^3=1$, the same equality follows from $A_{m/3}(1)=m/3$. All other factors in~\eqref{eq:ternary-R-factor} are units, proving the assertion, including at the two rational fixed points.
\end{proof}

\begin{corollary}[Fixed points in a boundary disk at $p=3$]\label{cor:ternary-boundary-count}
Under the hypotheses of Proposition~\ref{prop:ternary-boundary-valuation}, the number of fixed lifts of $\varepsilon$ modulo $3^k$ is
\[
 B_k(s)=
 \begin{cases}
  3^{k-1},&1\le k\le s+1,\\
  2\cdot3^s,&k\ge s+2.
 \end{cases}
\]
For $k\ge s+2$, these lifts form the disjoint union of the two congruence classes
\[
 x\equiv\varepsilon\pmod{3^{k-s}}
 \qquad\text{and}\qquad
 x\equiv-\varepsilon/2\pmod{3^{k-s}},
\]
each containing $3^s$ classes modulo~$3^k$.
\end{corollary}

\begin{proof}
For $x\in\varepsilon+3\Z_3$, put
\[
 a=\nu_3(x-\varepsilon),\qquad b=\nu_3(2x+\varepsilon).
\]
Both are at least~$1$. The identity
\[
 2(x-\varepsilon)-(2x+\varepsilon)=-3\varepsilon
\]
shows that they cannot both exceed~$1$, so $\min(a,b)=1$. Proposition~\ref{prop:ternary-boundary-valuation} therefore gives
\[
 \nu_3(T_n(x)-x)=a+b+s-1=\max(a,b)+s.
\]
This equality includes the exact roots under the convention $\nu_3(0)=+\infty$. If $k\le s+1$, every lift is fixed. If $k\ge s+2$, a lift is fixed if and only if $a\ge k-s$ or $b\ge k-s$, giving the two stated congruences. They are disjoint because the two rational roots differ by $3\varepsilon/2$, of valuation~$1$, while $k-s\ge2$. Each congruence determines $3^s$ classes modulo~$3^k$.
\end{proof}

\subsection{The fixed-point count at every level}
We now combine the local counts. The following elementary identity converts their powers of $p$ into the gcd terms of Theorem~\ref{thm:complete-fixed-count}.

\begin{lemma}[A gcd identity]\label{lem:gcd-saturation}
Let $M,m$ be positive integers with $p\nmid M$. For every $k\ge1$,
\[
 \gcd(m,Mp^{k-1})
 =\gcd(m,M)\,p^{\min(k-1,\nu_p(m))}.
\]
\end{lemma}

\begin{proof}
Write $m=p^v m_0$ with $p\nmid m_0$. Since $p\nmid M$, the left-hand side is $\gcd(m_0,M)p^{\min(k-1,v)}$, and $\gcd(m_0,M)=\gcd(m,M)$.
\end{proof}

For use here and for later applications to iterates, we collect the local conclusions in one statement.

\begin{lemma}[Local form of the fixed-point count]\label{lem:local-form}
Let $a_0\in\Fp$ be fixed by $T_n$, and put $s=\nu_p(n^2-1)$. Its number of fixed lifts modulo $p^k$ is
\[
 \begin{cases}
  1,&a_0\text{ nondegenerate},\\
  B_k(s),&a_0\text{ degenerate},\ p=3,\ a_0\in\{\pm1\},\\
  p^{\min(k-1,s)},&\text{every other degenerate case}.
 \end{cases}
\]
\end{lemma}

\begin{proof}
Combine Hensel's lemma, Lemmas~\ref{lem:local-nonboundary} and~\ref{lem:local-boundary}, and Corollary~\ref{cor:ternary-boundary-count}.
\end{proof}

\begin{proof}[Proof of Theorem~\ref{thm:complete-fixed-count}]
Put $L=p^{\min(k-1,s)}$. Lemma~\ref{lem:local-form} assigns one fixed lift to each nondegenerate residue and $L$ to each degenerate residue, except at a degenerate boundary residue when $p=3$. For $p=3$, this exception occurs precisely when $3\nmid n$. There are then $\beta$ fixed boundary residues, and Corollary~\ref{cor:ternary-boundary-count} gives
\[
 B_k(s)-L=
 \begin{cases}
  0,&k\le s+1,\\
  3^s,&k\ge s+2.
 \end{cases}
\]
Thus the total additional contribution is exactly $\Delta_{p,n,k}$ from~\eqref{eq:ternary-correction}. Summing over the fixed residues yields
\[
 N_k=(N_1-d)+dL+\Delta_{p,n,k},
\]
which proves~\eqref{eq:complete-fixed-lifting}.

To obtain~\eqref{eq:complete-fixed-count}, first suppose $s=0$. Then $d=0$ and $G_{i,k}=G_i$ for every~$i$, so Proposition~\ref{thm:N1} gives the required formula. If $s>0$, exactly one of $n-1$ and $n+1$ is divisible by~$p$. When $p\mid n-1$, Lemma~\ref{lem:gcd-saturation} gives
\[
 G_{1,k}=LG_1,\qquad G_{3,k}=LG_3,\qquad
 G_{2,k}=G_2,\qquad G_{4,k}=G_4.
\]
Since $d=(G_1+G_3)/2$, it follows that
\[
 N_1+d(L-1)
 =\frac{G_{1,k}+G_{2,k}+G_{3,k}+G_{4,k}-2\beta}{2}.
\]
When $p\mid n+1$, the same calculation uses $G_{2,k}=LG_2$, $G_{4,k}=LG_4$, and $d=(G_2+G_4)/2$. Adding $\Delta_{p,n,k}$ completes the proof.
\end{proof}

\begin{example}\label{ex:ternary-n2}
For $p=3$ and $n=2$, one has $N_1=d=s=\beta=1$, and Theorem~\ref{thm:complete-fixed-count} gives
\[
 N_1=1,\qquad N_2=3,\qquad N_k=6\quad(k\ge3).
\]
Indeed, $T_2(x)-x=(x-1)(2x+1)$. Its rational roots $1$ and $-1/2$ have the same reduction modulo $3$ but distinct reductions modulo~$9$. The six fixed classes modulo $27$ are the three lifts of each of
\[
 x\equiv1\pmod9,\qquad x\equiv-1/2\equiv4\pmod9.
\]

For the permutation degree $n=11$ at $p=3$, one has $N_1=d=3$, $s=\nu_3(120)=1$, and $\beta=2$. The formula gives $N_1=3$, $N_2=9$, and $N_k=15$ for $k\ge3$. These counts agree with~\cite[Table~IV]{LLTC2025}. The two boundary residues contribute the additional $2\cdot3=6$ fixed points at $k\ge3$, explaining why the stabilization threshold proved there for $p>3$ does not apply unchanged at~$p=3$.
\end{example}

\section{Signed orders, return multipliers, and first lifts}\label{sec:period-Fp}

We now pass from fixed points to exact periods. The source order determines a point's period modulo $p$. When $p\nmid n$, the derivative of the return map then determines the periods of its lifts modulo $p^2$. The connection between these two steps is a signed order modulo~$ep$, where $e$ is the source order.

\subsection{Finite-field periods}
We use the source orders of Section~\ref{sec:intro} and the signed order of Definition~\ref{def:signed-order}. The finite-field period criterion below is part of the known graph descriptions~\cite{Gassert2014,QP2019,HutzPatel2022}. We first record the elementary properties of signed orders needed both here and at higher prime powers.

\begin{samepage}
\begin{lemma}[Properties of signed order]\label{lem:cord-vs-ord}
Let $M\ge1$ and $\gcd(n,M)=1$. For every $r\ge1$,
\[
 n^r\equiv\pm1\pmod M
 \quad\Longleftrightarrow\quad \cord_M(n)\mid r.
\]
Consequently, $M'\mid M$ implies $\cord_{M'}(n)\mid\cord_M(n)$ for positive~$M'$.

For $M\ge2$, put $\rho=\operatorname{ord}_M(n)$. Then
\[
 \cord_M(n)=
 \begin{cases}
  \rho/2,&\rho\text{ even and }n^{\rho/2}\equiv-1\pmod M,\\
  \rho,&\text{otherwise}.
 \end{cases}
\]
If $M=p$ is an odd prime, then $\cord_p(n)=\operatorname{ord}_p(n^2)$.
\end{lemma}
\end{samepage}

\begin{proof}
For $M\ge3$, the signed order is the order of the class of $n$ in $(\Z/M\Z)^\times/\{\pm1\}$, so the first equivalence follows from the definition of the order of a group element. The cyclic subgroup $\langle n\rangle$ has order $\rho$. Its image in this quotient has order $\rho/2$ when $-1\in\langle n\rangle$, and order $\rho$ otherwise. Since a cyclic group of even order has the unique element $n^{\rho/2}$ of order~$2$, this gives the displayed formula. For $M=1,2$, the assertions follow from the conventions in Definition~\ref{def:signed-order}; when $M=2$, coprimality forces $n$ odd and $\rho=1$. Reducing the congruence $n^{\cord_M(n)}\equiv\pm1\pmod M$ modulo a divisor $M'$ proves the divisibility assertion.

Over $\Fp$, the equation $X^2=1$ has exactly the roots $\pm1$. Hence $n^{2r}\equiv1\pmod p$ is equivalent to $n^r\equiv\pm1\pmod p$, and taking the least such $r$ proves the last identity. This identity need not hold for composite moduli: $\cord_8(3)=2$, whereas $\operatorname{ord}_8(3^2)=1$.
\end{proof}

\begin{proposition}[Period from the source order]\label{thm:period}
Let $a\in\Fp$ have source order $e$. Then $a$ is periodic under $T_n$ if and only if $\gcd(n,e)=1$. In that case,
\[
 \operatorname{per}_p(a)=\cord_e(n).
\]
\end{proposition}

\begin{proof}
Choose a source parameter of order $e$. For every $j\ge1$, the composition identity and Lemma~\ref{lem:fixed-factor}, applied with degree $n^j$, give
\[
 T_n^{\circ j}(a)=a
 \quad\Longleftrightarrow\quad n^j\equiv\pm1\pmod e.
\]
This includes the boundary source orders $e=1,2$. If the congruence holds, then $\gcd(n,e)=1$. Conversely, coprimality ensures that such an exponent exists, and its least positive value is $\cord_e(n)$.
\end{proof}

Counting source parameters of each order now gives the number of points of any prescribed exact period.

\begin{corollary}[Direct periodic-point count]\label{cor:Pr-direct}
For $r\ge1$, the number of points of exact period $r$ under $T_n$ on $\Fp$ is
\begin{equation}\label{eq:Pr-direct}
 P_r^{(1)}=\beta[r=1]
 +\frac12\sum_{\substack{e\mid p-1,\ e\ge3\\
                         \gcd(n,e)=1,\ \cord_e(n)=r}}\varphi(e)
 +\frac12\sum_{\substack{e\mid p+1,\ e\ge3\\
                         \gcd(n,e)=1,\ \cord_e(n)=r}}\varphi(e).
\end{equation}
\end{corollary}

\begin{proof}
A cyclic source group contains $\varphi(e)$ parameters of order $e$ for each divisor $e$ of its order. For $e\ge3$, inversion pairs these parameters without fixed points, so Proposition~\ref{prop:two-source} gives $\varphi(e)/2$ distinct residues. Proposition~\ref{thm:period} selects exactly the orders that contribute to period~$r$. The two source groups intersect only at $\pm1$, so the sums have no overlap. The periodic boundary points are precisely the $\beta$ fixed boundary points, giving the term $\beta[r=1]$.
\end{proof}

\subsection{Return multipliers and the first lift}\label{sec:first-period-lift}
When $p\nmid n$, the return multiplier from Section~\ref{sec:intro} governs the first lift through the usual affine description of cycle lifting~\cite{desJardinsZieve2001,Nara2025}. For Chebyshev maps, its order can be computed from a signed congruence modulo~$ep$. The next lemma isolates the required arithmetic: when $e\ge3$, the congruences modulo $e$ and modulo $p$ must use the same sign.

\begin{samepage}
\begin{lemma}[Sign matching]\label{lem:first-lift-sign}
Assume $p\nmid n$, and let $e\ge1$ satisfy $\gcd(n,e)=1$ and $e\mid p-1$ or $e\mid p+1$. Put
\[
 t=\cord_e(n),\qquad \kappa=\cord_{ep}(n).
\]
Then $t\mid\kappa$. If $e\ge3$, let $\sigma\in\{\pm1\}$ be the unique sign such that $n^t\equiv\sigma\pmod e$. For every $u\ge1$,
\[
 \kappa\mid tu
 \quad\Longleftrightarrow\quad n^{tu}\equiv\sigma^u\pmod p.
\]
If $e=1$ or $e=2$, then $t=1$ and, for every $u\ge1$,
\[
 \kappa\mid u
 \quad\Longleftrightarrow\quad n^{2u}\equiv1\pmod p.
\]
In particular, $\kappa=t$ if and only if $n^t\equiv\sigma\pmod p$ in the nonboundary case, or $n^2\equiv1\pmod p$ in the boundary cases.
\end{lemma}
\end{samepage}

\begin{proof}
Lemma~\ref{lem:cord-vs-ord} gives $t\mid\kappa$. Since $e\mid p-1$ or $e\mid p+1$, we have $\gcd(e,p)=1$. For $e\ge3$, the congruence $n^{tu}\equiv\sigma^u\pmod e$ already holds, with its sign uniquely determined. The Chinese remainder theorem therefore gives
\[
 n^{tu}\equiv\pm1\pmod{ep}
 \quad\Longleftrightarrow\quad n^{tu}\equiv\sigma^u\pmod p.
\]
The first equivalence follows from Lemma~\ref{lem:cord-vs-ord}.

For $e=1,2$, the congruence modulo $e$ imposes no sign restriction; when $e=2$, $n$ is odd by coprimality. Thus $n^u\equiv\pm1\pmod{ep}$ is equivalent to $n^u\equiv\pm1\pmod p$, or to $n^{2u}\equiv1\pmod p$. Applying Lemma~\ref{lem:cord-vs-ord} proves the boundary assertion. Finally, taking $u=1$ gives the criteria for $\kappa=t$.
\end{proof}

\begin{proposition}[Return multiplier and signed order]\label{prop:return-multiplier}
Assume $p\nmid n$, let $a_0\in\Fp$ be periodic with source order $e$, and put
\[
 t=\cord_e(n),\qquad
 \lambda=(T_n^{\circ t})'(a_0)\in\Fp^\times.
\]
Then
\begin{equation}\label{eq:return-multiplier-order}
 \operatorname{ord}_{\Fp^\times}(\lambda)
 =\frac{\cord_{ep}(n)}{\cord_e(n)}.
\end{equation}
In particular, this quotient divides $p-1$.
\end{proposition}

\begin{proof}
Write $\kappa=\cord_{ep}(n)$. If $e\ge3$, let $\sigma$ be the sign of $n^t$ modulo $e$. Proposition~\ref{prop:degen}, applied to the return map $T_{n^t}$, gives $\lambda=\sigma n^t$ in~$\Fp$. Consequently, Lemma~\ref{lem:first-lift-sign} gives, for every $u\ge1$,
\[
 \lambda^u=1
 \quad\Longleftrightarrow\quad n^{tu}\equiv\sigma^u\pmod p
 \quad\Longleftrightarrow\quad \kappa\mid tu.
\]
Since $t\mid\kappa$, the least such $u$ is $\kappa/t$.

If $e=1,2$, then $t=1$ and Proposition~\ref{prop:degen}(b) gives $\lambda=n^2$ in~$\Fp$. The boundary assertion of Lemma~\ref{lem:first-lift-sign} shows that $\lambda^u=1$ precisely when $\kappa\mid u$, giving the same formula. In all cases $\lambda\ne0$ because $p\nmid n$, and its order divides $p-1$.
\end{proof}

The same calculation identifies which iterates have a degenerate fixed point at~$a_0$. This will allow the local counts from Section~\ref{sec:lifting} to be applied to iterates in Section~\ref{sec:higher-level}.

\begin{lemma}[Degeneracy for iterates]\label{lem:iterate-degen-sign}
Assume $p\nmid n$, let $a_0\in\Fp$ be periodic with source order $e$, and put $t=\cord_e(n)$. If $j\ge1$ and $t\mid j$, then $a_0$ is a degenerate fixed point of $T_{n^j}$ if and only if
\[
 n^j\equiv\pm1\pmod{ep},
\]
or equivalently, $\cord_{ep}(n)\mid j$.
\end{lemma}

\begin{proof}
Write $j=tu$ and $\lambda=(T_n^{\circ t})'(a_0)$. The return map fixes $a_0$, so the chain rule gives
\[
 (T_n^{\circ j})'(a_0)=\lambda^u.
\]
By Proposition~\ref{prop:return-multiplier}, this equals $1$ precisely when $\cord_{ep}(n)/t$ divides $u$, or equivalently, $\cord_{ep}(n)\mid j$. Lemma~\ref{lem:cord-vs-ord} translates the divisibility condition into the stated congruence.
\end{proof}

\begin{theorem}[First-lift period dichotomy]\label{thm:first-period-lift}
Let $p$ be an odd prime, let $n\ge2$ with $p\nmid n$, and let $a_0\in\Fp$ be periodic with source order $e$. Put
\[
 t=\operatorname{per}_p(a_0)=\cord_e(n),\qquad
 \kappa=\cord_{ep}(n).
\]
Then $t\mid\kappa$, every lift of $a_0$ modulo $p^2$ is periodic, and its exact period is either $t$ or $\kappa$. More precisely:
\begin{enumerate}[label=\textup{(\alph*)},nosep]
\item If $\kappa=t$, all $p$ lifts have exact period~$t$.
\item If $\kappa>t$, exactly one lift has exact period $t$, and the other $p-1$ lifts have exact period~$\kappa$.
\end{enumerate}
\end{theorem}

\begin{proof}
Set $f=T_n^{\circ t}=T_{n^t}$ and $\lambda=f'(a_0)$. Proposition~\ref{prop:return-multiplier} gives
\[
 L:=\operatorname{ord}_{\Fp^\times}(\lambda)=\kappa/t.
\]
Lemma~\ref{lem:local-first-lift}, applied with degree $n^t$, supplies an exact fixed point $\alpha\in\Z_p$ of $f$ reducing to~$a_0$. Write the $p$ lifts as $\alpha+pu$ modulo $p^2$, with $u\in\Fp$. Taylor expansion gives
\[
 f(\alpha+pu)\equiv\alpha+p\lambda u\pmod{p^2}.
\]
Thus the return map acts on the coordinate $u$ by multiplication by~$\lambda$. Since $\lambda\ne0$, every coordinate is periodic. If $L=1$, all coordinates are fixed. If $L>1$, only $u=0$ is fixed, and every nonzero coordinate has exact period~$L$ under this action.

Any iterate of $T_n$ fixing a lift must have exponent divisible by $t$, because its reduction fixes $a_0$. Hence the exact period of a lift under $T_n$ is $t$ times its period under the return map. The asserted periods and multiplicities follow.
\end{proof}

The theorem counts points above one residue. To count cycles, we consider the lifts of an entire cycle modulo~$p$.

\begin{corollary}[First-lift cycle counts]\label{cor:first-period-cycle-count}
Assume $p\nmid n$. Let $\mathcal C$ be a cycle of $T_n$ on $\Fp$ of length $t$, and let $e$ be the source order of one of its points. Put $\kappa=\cord_{ep}(n)$ and $L=\kappa/t$. The residue classes modulo $p^2$ that reduce to points of $\mathcal C$ form
\begin{enumerate}[label=\textup{(\alph*)},nosep]
\item $p$ cycles of length $t$ if $L=1$;
\item one cycle of length $t$ and $(p-1)/L$ cycles of length $\kappa$ if $L>1$.
\end{enumerate}
\end{corollary}

\begin{proof}
If $\zeta$ is a source parameter of one point of $\mathcal C$, the successive parameters may be chosen as $\zeta,\zeta^n,\ldots,\zeta^{n^{t-1}}$. Since $\gcd(n,e)=1$, they all have order $e$. Theorem~\ref{thm:first-period-lift} therefore gives the same point multiplicities above each of the $t$ residues.

For $L=1$, the $pt$ lifted points all have period $t$, giving $p$ cycles. For $L>1$, there are $t$ lifted points of period $t$ and $(p-1)t$ of period $\kappa=tL$. Dividing by the respective periods gives the stated cycle counts. Their integrality also follows from $L\mid p-1$ in Proposition~\ref{prop:return-multiplier}.
\end{proof}

\begin{samepage}
\begin{example}[An even-degree example]
Take $T_2(x)=2x^2-1$ and $p=7$. The fixed residue $a_0=3$ has source parameters $2,4\in\mathbb F_7^\times$, both of order $e=3$. Thus
\[
 t=\cord_3(2)=1,\qquad
 \kappa=\cord_{21}(2)=6.
\]
The return multiplier is $T_2'(3)=12\equiv5\pmod7$, of order~$6$. Among the seven lifts of $3$ modulo $49$, the unique fixed lift is $24$, and the other six form the cycle
\[
 3\longmapsto17\longmapsto38\longmapsto45
  \longmapsto31\longmapsto10\longmapsto3.
\]
This gives one cycle of length $1$ and one of length $6$, as predicted by the corollary.
\end{example}
\end{samepage}

\section{Exact periods at higher prime powers}\label{sec:higher-level}

We now determine the exact periods of lifts modulo $p^k$ and count the points and cycles of each period. When $p\nmid n$, the growth of the signed orders $\cord_{ep^q}(n)$ gives these counts above every periodic residue, except at the boundary points for $p=3$. We treat those points separately using the boundary fixed-point count from Section~\ref{sec:lifting}. When $p\mid n$, each periodic residue has a unique periodic lift, with the same period. These explicit counts refine the general cycle-lifting alternatives in Nara~\cite[Corollary~4.12]{Nara2025}.

\subsection{Growth of signed orders}

\begin{proposition}[Growth of signed orders]
\label{prop:signed-order-growth}
Let $p$ be an odd prime, let $n\ge2$ with $p\nmid n$, and let $e\ge1$ satisfy
\[
  \gcd(n,e)=1,\qquad e\mid p-1\ \text{or}\ e\mid p+1.
\]
Put
\[
  c_q=\cord_{ep^q}(n)\quad(q\ge0),\qquad
  R=\cord_p(n)=\operatorname{ord}_p(n^2),\qquad
  w=\nup(n^{2R}-1).
\]
Then $w\ge1$, $p\nmid c_1$, and
\begin{equation}\label{eq:signed-order-growth}
  c_q=
  \begin{cases}
    c_0,&q=0,\\
    c_1,&1\le q\le w,\\
    c_1p^{q-w},&q>w.
  \end{cases}
\end{equation}
\end{proposition}

\begin{proof}
Lemma~\ref{lem:cord-vs-ord} gives $c_q\mid c_{q+1}$ and $R\mid c_1$. Also, $w\ge1$ because $p\mid n^{2R}-1$. The condition on $e$ implies $\varphi(e)<p$: this is immediate if $e<p$, while the only remaining possibility is $e=p+1$, for which $\varphi(e)\le(p+1)/2<p$. Since $\gcd(e,p)=1$,
\[
  c_1\mid\operatorname{ord}_{ep}(n)\mid\varphi(e)(p-1),
\]
so $p\nmid c_1$. Write $c_1=Rh$. The lifting-the-exponent identity~\eqref{eq:LTE}, with $p\nmid h$, gives
\[
  \nup(n^{2c_1}-1)
  =\nup\bigl((n^{2R})^h-1\bigr)=w.
\]
Choose $\sigma\in\{\pm1\}$ with $n^{c_1}\equiv\sigma\pmod{ep}$. Since $p$ is odd, $n^{c_1}+\sigma$ is a $p$-adic unit, and hence $\nup(n^{c_1}-\sigma)=w$. Applying~\eqref{eq:LTE} to $\sigma n^{c_1}$ yields
\begin{equation}\label{eq:signed-LTE}
  \nup(n^{c_1p^a}-\sigma)=w+a
  \qquad(a\ge0),
\end{equation}
where the sign remains $\sigma$ because $p^a$ is odd.

For $1\le q\le w$, we have $n^{c_1}\equiv\sigma\pmod{ep^q}$, so $c_q\mid c_1$. Together with $c_1\mid c_q$, this gives $c_q=c_1$.
For $q>w$, equation~\eqref{eq:signed-LTE} shows that
\[
  n^{c_1p^{q-w}}\equiv\sigma\pmod{ep^q},
  \qquad c_1\mid c_q\mid c_1p^{q-w}.
\]
Thus $c_q=c_1p^a$ for some $0\le a\le q-w$. The sign in $n^{c_q}\equiv\pm1\pmod{ep^q}$ must be $\sigma$, as is seen by reducing modulo $p$. Equation~\eqref{eq:signed-LTE} then gives $w+a\ge q$, so $a=q-w$.
\end{proof}

Coprimality of $e$ and $p$ alone is insufficient for this formula. For $p=3$, $e=7$, and $n=2$, one has $c_0=3$, $c_1=c_2=6$, and $w=1$, whereas~\eqref{eq:signed-order-growth} would give $c_2=18$.

\subsection{Period counts away from the boundary at \texorpdfstring{$p=3$}{p=3}}

\begin{theorem}[Exact periods of lifts]
\label{thm:explicit-period-tower}
\begin{samepage}
Let $p$ be an odd prime and $n\ge2$ with $p\nmid n$. Let $a_0\in\Fp$ be periodic with source order $e$, and let $k\ge1$. Assume either $p\ge5$ or $a_0\notin\{\pm1\}$. Put
\[
  t=\cord_e(n),\qquad \kappa=\cord_{ep}(n),\qquad
  R=\cord_p(n),\qquad
  w=\nup(n^{2R}-1),\qquad
  m=\min(k-1,w).
\]
Every lift of $a_0$ modulo $p^k$ is periodic, and the exact periods and their multiplicities are as follows:
\[
\begin{array}{c c c}
\toprule
  \text{Case}&\text{Exact period}&\text{Number of lifts}\\ \midrule
  \kappa=t&t&p^m\\
           &tp^i&(p-1)p^{w+i-1}\\ \midrule
  \kappa>t&t&1\\
           &\kappa&p^m-1\\
           &\kappa p^i&(p-1)p^{w+i-1}\\
\bottomrule
\end{array}
\]
Here $1\le i\le k-1-w$; empty ranges and entries with zero multiplicity are omitted.
\end{samepage}
\end{theorem}

\begin{proof}
The case $k=1$ is immediate. Assume $k\ge2$, and write $c_q=\cord_{ep^q}(n)$. For each positive multiple $j$ of $t$, let $A_j$ be the set of lifts of $a_0$ modulo $p^k$ fixed by $T_n^{\circ j}=T_{n^j}$, and put
\[
  q_k(j)=\max\{q\in\{0,\ldots,k-1\}:c_q\mid j\}.
\]
We first show that $|A_j|=p^{q_k(j)}$.

By Lemma~\ref{lem:iterate-degen-sign}, $a_0$ is a degenerate fixed point of $T_{n^j}$ if and only if $\kappa\mid j$, or equivalently, $q_k(j)\ge1$. If $q_k(j)=0$, Hensel's lemma gives $|A_j|=1$. Otherwise, choose the unique sign $\sigma\in\{\pm1\}$ for which $n^j\equiv\sigma\pmod{ep}$. For $1\le q\le k-1$, Lemma~\ref{lem:cord-vs-ord} gives
\[
  c_q\mid j
  \iff n^j\equiv\sigma\pmod{ep^q}
  \iff q\le\nup(n^j-\sigma).
\]
Since $n^j+\sigma$ is a $p$-adic unit,
\[
  q_k(j)=\min\{k-1,\nup(n^{2j}-1)\}.
\]
The applicable degenerate case of Lemma~\ref{lem:local-form}, with degree $n^j$, therefore gives $|A_j|=p^{q_k(j)}$, as claimed.

We next show that the values $c_0,\ldots,c_{k-1}$ include every possible exact period. Since $q_k(c_{k-1})=k-1$, the iterate $T_n^{\circ c_{k-1}}$ fixes all $p^{k-1}$ lifts, so every lift is periodic. Let a lift have exact period $j$. Reduction modulo $p$ gives $t\mid j$. Set $h=q_k(j)$. Then $c_h\mid j$ and $q_k(c_h)=h$: if $c_q\mid c_h$ for some $q>h$, then $c_q\mid j$, contrary to the definition of $h$. Consequently,
\[
  A_{c_h}\subseteq A_j,\qquad |A_{c_h}|=|A_j|=p^h,
\]
and these sets are equal. The lift is therefore fixed by $T_n^{\circ c_h}$, so $j\mid c_h$. Hence $j=c_h$.

It remains to count the points of each exact period. Proposition~\ref{prop:signed-order-growth} applies because the source order satisfies $\gcd(n,e)=1$ and $e\mid p-1$ or $e\mid p+1$. It gives
\[
  c_0=t,\qquad c_q=\kappa\quad(1\le q\le m),\qquad
  c_{w+i}=\kappa p^i\quad(1\le i\le k-1-w).
\]
The fixed sets for the distinct values in this divisibility chain are nested. If $\kappa=t$, there are $p^m$ points fixed by $T_n^{\circ t}$. If $\kappa>t$, there is one such point, and $p^m-1$ further points have exact period $\kappa$. For each later value, subtracting the preceding fixed-point count gives
\[
  p^{w+i}-p^{w+i-1}=(p-1)p^{w+i-1}.
\]
Because no other exact periods occur, these differences give the stated multiplicities.
\end{proof}

Continuing the example at the end of Section~\ref{sec:period-Fp}, take $p=7$, $n=2$, and $a_0=3$. Here $e=3$, $t=1$, $\kappa=6$, $R=3$, and $w=1$. At level $k=4$, the $343$ lifts of $a_0$ split into $1$, $6$, $42$, and $294$ points of exact periods $1$, $6$, $42$, and $294$, respectively.

\begin{corollary}[Cycle counts above a base cycle]
\label{cor:cycle-count-tower}
Under the hypotheses of Theorem~\ref{thm:explicit-period-tower}, consider the full base cycle of length $t$, and put $L=\kappa/t$. Its lifts consist of the following cycles. If $L=1$, there are $p^m$ cycles of length $t$. If $L>1$, there is one cycle of length $t$, together with $(p^m-1)/L$ cycles of length $\kappa$. In either case, for each $1\le i\le k-1-w$, there are also
\[
  \frac{(p-1)p^{w-1}}{L}
  \quad\text{cycles of length }\kappa p^i.
\]
\end{corollary}

\begin{proof}
The source order is constant along the base cycle, as observed in Corollary~\ref{cor:first-period-cycle-count}. Thus the point counts in Theorem~\ref{thm:explicit-period-tower} are the same above each of its $t$ points. Multiplying each count by $t$ and dividing by the corresponding exact period gives the number of cycles. Proposition~\ref{prop:return-multiplier} gives $L\mid p-1$, so the displayed quotients are integers.
\end{proof}

Under the permutation hypothesis $\gcd(n,p^3-p)=1$ of Li--Lu--Tan--Chen~\cite{LLTC2025}, with $p>3$, their parameters correspond to
\[
  N_s=t,\qquad l_s=L=\kappa/t,\qquad
  w=\nup(n^{2\operatorname{ord}_p(n^2)}-1).
\]
The identification of $l_s$, the order of the return multiplier, follows from Proposition~\ref{prop:return-multiplier}. With these substitutions, Corollary~\ref{cor:cycle-count-tower} agrees with their Proposition~7. When $L=1$, their two initial families have the same period and together give $p^m$ cycles of length $t$.

At $p=3$, the sole nonboundary residue is $a_0=0$, with source order $4$. Under the standing assumption $3\nmid n$, it is periodic exactly when $n$ is odd. In that case $t=R=1$, $w=\nu_3(n^2-1)$, and Corollary~\ref{cor:cycle-count-tower} gives the cycle counts in Lu--Dai--Li~\cite[Proposition~4]{LuDaiLi2026}, with their $l_0=\operatorname{ord}_{\Fp^\times}(T_n'(0))$ equal to $L$.

\subsection{Boundary periods for \texorpdfstring{$p=3$}{p=3}}

For $\varepsilon\in\{\pm1\}$, write
\[
  D_{\varepsilon,k}
  =\{x\in\Z/3^k\Z:x\equiv\varepsilon\pmod3\}.
\]

\begin{theorem}[Exact periods in a boundary disk]
\label{thm:ternary-boundary-periods}
Let $n\ge2$, $3\nmid n$, and $k\ge1$. Let $\varepsilon\in\{\pm1\}$ satisfy $T_n(\varepsilon)=\varepsilon$, and put $s=\nu_3(n^2-1)$. If $k\le s+1$, every point of $D_{\varepsilon,k}$ is fixed.

If $k\ge s+2$, every point of $D_{\varepsilon,k}$ is periodic, and its exact period divides $3^{k-s-1}$. The multiplicities are
\[
\begin{array}{c c c}
\toprule
  \text{Exact period}&\text{Number of points}&\text{Range}\\ \midrule
  1&2\cdot3^s&\\
  3^q&4\cdot3^{s+q-1}&1\le q\le k-s-2\\
  3^{k-s-1}&3^{k-2}&\\
\bottomrule
\end{array}
\]
with the middle row omitted when its range is empty.
\end{theorem}

\begin{proof}
The disk $D_{\varepsilon,k}$ is invariant because $\varepsilon$ is fixed modulo $3$. The case $k\le s+1$ follows from Corollary~\ref{cor:ternary-boundary-count}. Now suppose $k\ge s+2$. The lifting-the-exponent identity~\eqref{eq:LTE} gives
\begin{equation}\label{eq:ternary-iterate-valuation}
  \nu_3(n^{2j}-1)=s+\nu_3(j)\qquad(j\ge1).
\end{equation}
Apply Corollary~\ref{cor:ternary-boundary-count} to $T_n^{\circ3^q}=T_{n^{3^q}}$. Its fixed-point count in the disk is
\[
  \bigl|\Fix(T_n^{\circ3^q},D_{\varepsilon,k})\bigr|
  =\begin{cases}
    2\cdot3^{s+q},&0\le q\le k-s-2,\\
    3^{k-1},&q\ge k-s-1.
  \end{cases}
\]
In particular, $T_n^{\circ3^{k-s-1}}$ fixes the whole disk. Thus every point is periodic and its period divides $3^{k-s-1}$.

There are $2\cdot3^s$ fixed points. For $1\le q\le k-s-2$, the number of points of exact period $3^q$ is the difference between the fixed-point counts for the iterates $3^q$ and $3^{q-1}$, namely
\[
  2\cdot3^{s+q}-2\cdot3^{s+q-1}=4\cdot3^{s+q-1}.
\]
The remaining points have exact period $3^{k-s-1}$, and their number is
\[
  3^{k-1}-2\cdot3^{k-2}=3^{k-2}.
\]
\end{proof}

For $k\ge s+2$, dividing these point counts by the corresponding periods gives $2\cdot3^s$ fixed points, $4\cdot3^{s-1}$ cycles of each length $3^q$ with $1\le q\le k-s-2$, and $3^{s-1}$ cycles of length $3^{k-s-1}$ in one boundary disk. For example, $T_2$ on the disk $x\equiv1\pmod3$ modulo $81$ has six fixed points, four cycles of length $3$, and one cycle of length $9$.

When $n$ is odd, both boundary points are fixed. For $k\ge s+2$, summing over their two disks gives $4\cdot3^s$ fixed points, $8\cdot3^{s-1}$ cycles of each length $3^q$ with $1\le q\le k-s-2$, and $2\cdot3^{s-1}$ cycles of length $3^{k-s-1}$. These agree with Lu--Dai--Li~\cite[Proposition~5]{LuDaiLi2026}, whose parameter $w$ is our $s$. When $n$ is even, only $\varepsilon=1$ satisfies $T_n(\varepsilon)=\varepsilon$, and the theorem applies to that disk.

\subsection{\texorpdfstring{The case $p\mid n$}{The case p divides n}}

When $p\mid n$, the derivative $T_n'=nU_{n-1}$ vanishes modulo $p$. The following result is the zero-multiplier case of Nara~\cite[Corollary~4.12(1)]{Nara2025}; here it has a direct proof by Hensel's lemma.

\begin{proposition}[Periodic points when $p\mid n$]
\label{prop:p-divides-n}
Let $p$ be an odd prime and let $n\ge2$ satisfy $p\mid n$. For every $k\ge1$, reduction modulo $p$ gives a bijection from the periodic points of $T_n$ on $\Z/p^k\Z$ to those on $\Fp$, preserving exact periods. Consequently,
\[
  P_r^{(k)}=P_r^{(1)}\qquad(k,r\ge1).
\]
\end{proposition}

\begin{proof}
For every $j\ge1$, the polynomial $g_j(x)=T_{n^j}(x)-x$ satisfies
\[
  g_j'(x)=n^jU_{n^j-1}(x)-1\equiv-1\pmod p.
\]
Hensel's lemma therefore gives a unique lift modulo $p^k$ of each root of $g_j$ modulo $p$.

Let $a_0$ have exact period $t$ modulo $p$, and let $\alpha_k$ be its unique lift fixed by $T_n^{\circ t}$. Its exact period divides $t$ and is a multiple of $t$ by reduction, so it equals $t$. If another lift $x$ of $a_0$ is periodic with exact period $r$, then $t\mid r$. Both $x$ and $\alpha_k$ are fixed by $T_n^{\circ r}$, so uniqueness for $g_r$ gives $x=\alpha_k$. Thus each periodic residue has exactly one periodic lift, with the same period.
\end{proof}

\begin{corollary}[The case $n=p$]\label{cor:fermat-periodic}
Let $p$ be an odd prime and $k\ge1$. Every periodic point of $T_p$ modulo $p^k$ is fixed, and
\[
  P_1^{(k)}=p,\qquad P_r^{(k)}=0\quad(r>1).
\]
\end{corollary}

\begin{proof}
By the Chebyshev--Fermat congruence in Corollary~\ref{cor:fermat}, $T_p$ fixes every element of $\Fp$. Apply Proposition~\ref{prop:p-divides-n}.
\end{proof}

\subsection{Global exact-period counts}

The local results describe the periods above each base cycle. For the total number of points of a prescribed exact period, M\"obius inversion gives a direct formula from Theorem~\ref{thm:complete-fixed-count}.

\begin{corollary}[M\"obius formula at arbitrary levels]
\label{cor:Prk-mobius}
For every odd prime $p$, $n\ge2$, $k\ge1$, and $r\ge1$,
\begin{equation}\label{eq:Prk-mobius}
  P_r^{(k)}
  =\sum_{j\mid r}\Mob(r/j)\,N_{k,p}(n^j).
\end{equation}
Here $N_{k,p}(m)$ is given by Theorem~\ref{thm:complete-fixed-count}, with the valuation $\nu_p(m^2-1)$, parity factor $\gcd(m-1,2)$, and boundary correction evaluated separately for each degree $m=n^j$.
\end{corollary}

\begin{proof}
A point is fixed by $T_n^{\circ j}=T_{n^j}$ if and only if its exact period divides $j$. Thus
\[
  N_{k,p}(n^j)=\sum_{\ell\mid j}P_\ell^{(k)},
\]
and M\"obius inversion gives~\eqref{eq:Prk-mobius}.
\end{proof}

For each fixed period, these counts eventually become constant. The next corollary gives a sufficient level for stabilization.

\begin{corollary}[Eventual stabilization]\label{cor:Pr-stabilization}
Let $p$ be an odd prime, $n\ge2$, and $r\ge1$. Put
\[
  k_0(r)=\nup(n^{2r}-1)+1+[p=3][3\nmid n].
\]
Then
\[
  P_r^{(k)}=P_r^{(k_0(r))}\qquad(k\ge k_0(r)).
\]
\end{corollary}

\begin{proof}
For $j\mid r$, put $s_j=\nup(n^{2j}-1)$. Since $n^{2j}-1$ divides $n^{2r}-1$, we have $s_j\le\nup(n^{2r}-1)$. By Lemma~\ref{lem:gcd-saturation}, the gcd terms in Theorem~\ref{thm:complete-fixed-count}, applied with degree $n^j$, are constant for $k\ge s_j+1$. If $p=3$ and $3\nmid n$, the boundary correction is constant for $k\ge s_j+2$; otherwise it is absent. Hence every summand in~\eqref{eq:Prk-mobius} is constant for $k\ge k_0(r)$, proving the assertion. When $p\mid n$, the valuation and the bracket term both vanish, so $k_0(r)=1$, in agreement with Proposition~\ref{prop:p-divides-n}.
\end{proof}

These results give the possible periods and the numbers of points and cycles of each period modulo every odd prime power. Classifying the trees of nonperiodic points entering the cycles for arbitrary degrees remains a separate problem. At $p=2$, division by $2$ is unavailable in the residue ring, so the averaging parametrization requires a different local analysis. For results at powers of $2$, see Fan--Liao~\cite{FanLiao2016} and Lu--Dai--Li~\cite{LuDaiLi2026}.

\section*{Declaration of competing interest}
The authors declare no competing interests.

\section*{Funding}
This research did not receive any specific grant from funding agencies in the
public, commercial, or not-for-profit sectors.

\section*{Data availability}
No external datasets were used. All results are proved in the manuscript.

\section*{Declaration of generative AI and AI-assisted technologies in the manuscript preparation process}
During the preparation of this work, the authors used ChatGPT and Codex
(OpenAI) and Claude (Anthropic) for language editing, manuscript organization,
and reference checking. These tools also provided suggestions for intermediate
mathematical derivations and helped draft code for direct enumeration of fixed
and periodic points modulo prime powers. The computations were used for
internal consistency checks and do not form part of the proofs. The authors
reviewed and edited all AI-assisted material, independently verified the
mathematical statements and proofs, and take full responsibility for the
content of the manuscript.

\end{document}